\documentclass[matann]{svjour}
\usepackage{latexsym, amsmath, amssymb, amsfonts, epsfig, subfigure, afterpage}

\renewcommand{\SS}{\mathcal S}
\newcommand{\FF}{\mathcal{F}}
\newcommand{\DD}{\mathcal{D}}

\newcommand{\Z}{{\mathbb Z}}

\newcommand{\R}{{\mathbb R}}
\newcommand{\C}{{\mathbb C}}

\newcommand{\II}{{\mathbb I}}
\newcommand{\T}{{\mathbb T}}  

\newcommand{\sgn}{\operatorname{sgn}}

\renewcommand{\th}{\ensuremath{\theta}}

\newcommand{\lam}{\ensuremath{\lambda}}

\renewcommand{\a}{\ensuremath{\alpha}}
\newcommand{\g}{\ensuremath{\gamma}}
\renewcommand{\b}{\ensuremath{\beta}}

\newcommand{\eps}{\ensuremath{\varepsilon}}
\newcommand{\e}{\ensuremath{\varepsilon}}

\newcommand{\hatH}{\widehat{H}}
\newcommand{\hatF}{\widehat{F}}
\newcommand{\hatI}{\widehat{I}}

\renewcommand{\(}{\left(}
\renewcommand{\)}{\right)}
\newcommand{\pfrac}[2]{\(\frac{#1}{#2}\)}
\newcommand{\ds}{\displaystyle}

\def\nts{\negthickspace}

\newcommand{\RHF}{RH$_F$}             
\newcommand{\be}{\begin{equation}}
\newcommand{\ee}{\end{equation}}
\newcommand{\benn}{\begin{equation*}} 
\newcommand{\eenn}{\end{equation*}}

\newcommand{\ms}{\medskip}
\numberwithin{equation}{section}

\begin{document}

\title
{On the distribution of imaginary parts of zeros
of the Riemann zeta function, II}

\titlerunning{Imaginary parts of zeros of $\zeta(s)$, II}

\author{Kevin Ford \thanks{The first author is supported by 
National Science Foundation Grant DMS-0555367} \and 
K. Soundararajan \thanks{The second author is partially 
supported by the National 
Science Foundation and the American Institute of Mathematics (AIM)} \and
Alexandru Zaharescu \thanks{The third author is supported by
 National Science Foundation Grant DMS-0456615}}

\institute{\textsc{Kevin Ford and Alexandru Zaharescu} \at 
Department of Mathematics, 1409 West Green Street, University
of Illinois at Urbana-Champaign, Urbana, IL 61801, USA \and 
\textsc{K. Soundararajan} \at
 Department of Mathematics, 450 Serra Mall, Bldg. 380, Stanford 
University, Stanford, CA 94305, USA}


\maketitle

\begin{abstract}
We continue our investigation of the distribution of the fractional 
parts of $\a \g$, where
$\a$ is a fixed non-zero real number and $\g$ runs over the imaginary
parts of the non-trivial zeros of the Riemann zeta function.  
We establish 
some connections to Montgomery's 
pair correlation function and the distribution of
primes in short intervals.  We also discuss analogous results for a
more general $L$-function.
\end{abstract}

%
\section{Introduction and Statement of Results}
%

In this paper we continue the study of 
the distribution of the fractional parts $\{\a \g\}$ 
initiated by the first and third authors in \cite{FZ}, where $\a$ is 
a fixed positive real number and $\g$ runs over the positive 
ordinates of zeros of the Riemann zeta function $\zeta(s)$.  
We extend and generalize the results from \cite{FZ} in several
directions, establishing connections between these fractional parts,
the pair correlation of zeros of $\zeta(s)$ and the distribution of
primes in short intervals.
It is known \cite{Hl} that for any fixed $\a$, 
the fractional parts $\{\a \g\}$ are uniformly distributed $\pmod 1$.  
That is, for all continuous functions 
$f:\T \to \C$, as $T\to \infty$  we have 
\be\label{uniform}
\sum_{0< \g \le T} f(\a\g)  = N(T) \int_\T f(x) dx
+ o(N(T)).
\ee
Here $\T$ is the torus $\R/\Z$ and $N(T)$ denotes 
the number of ordinates $0 < \g \le T$; it is 
well-known that 
\be\label{NT}
N(T) = \frac{T}{2\pi} \log \frac{T}{2\pi e} +O(\log T).
\ee  
We are interested in the lower order terms in the asymptotic 
\eqref{uniform}.  For a general 
continuous function $f$ the asymptotic \eqref{uniform} can be attained
arbitrarily slowly so 
that no improvement of the error term there is possible.  
But if we assume that $f$ has 
nice smoothness properties then we can isolate a second main term of 
size about $T$.  
More precisely, we define the function $g_{\a}:\T \to \C$ as follows.  If 
$\a$ is not a rational multiple of $\frac{\log p}{2\pi}$ 
for some prime $p$, then 
$g_{\a}$ is identically zero.  If $\a= \frac{a}{q} \frac{\log p}{2\pi} 
$ for some 
rational number $a/q$ with $(a,q)=1$ then we set 
\be\label{galpha}
g_{\a}(x) = -\frac{\log p}{\pi} \Re \sum_{k=1}^{\infty} 
\frac{e^{-2\pi i qkx}}{p^{ak/2}} = - \frac{(p^{a/2} \cos 2\pi q x -
  1)\log p}{\pi(p^a - 2p^{a/2}\cos 2\pi q x + 1)}.
\ee
Then, we expect (for suitable $f$) that as $T\to \infty$ 
\be\label{FZ}
\sum_{0< \g \le T} f(\a \g) = N(T) \int_\T f(x) dx + 
T \int_\T f(x) g_{\a}(x) dx + o(T).
\ee
As remarked above, certainly \eqref{FZ} does not hold for all
continuous functions $f$.  
In Corollary 2 of \cite{FZ}, it is shown that \eqref{FZ} holds for 
all $f\in C^2 (\T)$, and if the Riemann Hypothesis
(RH) is true then \eqref{FZ} holds for all absolutely 
continuous functions $f$ (see Corollary 5 there).  
Moreover it is conjectured 
there (see Conjecture A there) that \eqref{FZ} does hold when $f$ is the 
characteristic function of an interval in $\T$. 

%
%

\begin{conjecture}\label{conj1}
Let $\II$ be an interval of $\T$.  Then 
$$
\sum_{\substack{ 0< \g \le T \\ \{\a \g\} \in \II}} 
1 = |\II| N(T) + T \int_{\II} g_{\a}(x) dx + o(T),
$$
uniformly in $\II$.
\end{conjecture}

We define the discrepancy of the sequence $\{\a \g\}$ (for $0<\g \le T$) 
as
$$
D_{\a}(T) = \sup_{\II} \Big| \frac{1}{N(T)} \sum_{\substack{ 
0 < \g \le T \\ \{\a \g\} \in \II}} 1 - |\II|\Big|,
$$
where the supremum is over all intervals $\II$ of $\T$.  
Unconditionally, Fujii \cite{F76} proved that $D_\a(T) \ll
\frac{\log\log T}{\log T}$ for every $\a$.  On RH, Hlawka \cite{Hl}
  showed that $D_\a(T) \ll \frac{1}{\log T}$, which is best possible
  for $\a$  of the form $\frac{a}{q} \frac{\log p}{2\pi}$ (\cite{FZ},
  Corollary 3).  Conjecture 1 
clearly implies the following conjecture for the discrepancy (see
Conjecture A and Corollary 6 of \cite{FZ}).

%
%

\begin{conjecture}\label{conj2} 
We have 
$$
D_{\a}(T) = \frac{T}{N(T)} \sup_{\II} \Big| \int_{\II} g_{\a}(x) dx \Big| 
+ o\Big(\frac{1}{\log T} \Big).
$$
\end{conjecture}
 
Even assuming RH, we are unable to 
establish Conjectures \ref{conj1} and \ref{conj2}.  
We show here some weaker results towards these conjectures, 
and how these conjectures 
would follow from certain natural assumptions on the zeros of 
$\zeta(s)$, or the distribution of prime numbers. 


\begin{theorem}\label{theorem1}
(i) We have unconditionally
$$
D_{\a}(T) \ge \frac{T}{N(T)}  \sup_{\II} \Big| \int_{\II} g_{\a}(x) dx \Big| 
 + o\Big(\frac{1}{\log T} \Big).
 $$

(ii) Assuming RH, for any interval $\II$ of $\T$ we have 
$$
\Big| \sum_{\substack{ 0<\g \le T \\ \{\a \g\} \in \II}} 1 
- |\II| N(T) - T\int_{\II} g_{\a}(x) dx \Big| \le  \frac{\a}{2} T + o(T).
$$

\end{theorem}

The left side of \eqref{uniform} depends strongly on the behavior of
the sums $\sum_{0<\g \le T} x^{i\g}$.  

\begin{conjecture}\label{conj3}
Let $A >1$ be a fixed real number.  Uniformly for all 
$\frac{T^2}{(\log T)^{5}} \le x\le T^{A}$ we have
\be\label{sumxg}
\sum_{0<\g \le T} x^{i\g} = o(T). 
\ee
\end{conjecture}

\begin{theorem}\label{123}
Assume RH.  Then Conjecture \ref{conj3} implies Conjectures \ref{conj1} and
\ref{conj2}.
\end{theorem}

{\bf Remarks}.
Assuming RH, 
\eqref{sumxg} holds for $x\to \infty$ and $x = o (T^2/\log^{4} T)$ as $T\to\infty$  by
uniform versions of Landau's formula for $\sum_{0<\g \le T} x^{\rho}$
\cite{La}.
For example, Lemma 1 of \cite{FZ} implies, for $x>1$ and $T\ge 2$,
that (unconditionally)
\be\label{Landau}
\sum_{0<\g \le T} x^\rho = 
-\frac{\Lambda(n_x)}{2\pi} \frac{e^{iT\log(x/n_x)}-1}
{i\log(x/n_x)}
+ O\( x \log^2 (Tx) + \frac{\log T}{\log x}\),
\ee
where $n_x$ is the nearest prime power to $x$, and the main term is to
be interpreted as $-T \frac{\Lambda(x)}{2\pi}$ if $x=n_x$.  This main
term is always $\ll T \log x$.
On RH, divide both sides of \eqref{Landau} by $x^{1/2}$ to obtain
\eqref{sumxg}. 
Unconditionally, one can use Selberg's zero-density estimate to deduce
$$
\Big| \sum_{0<\g\le T} (x^{i\g} - x^{\rho-1/2}) \Big| \ll
\frac{T\log^2 (2x)}{\log T};
$$
see e.g. (3.8) of \cite{FZ}.  This gives \eqref{sumxg} when $\log x =
o(\sqrt{\log T})$.

We next relate Conjecture \ref{conj3} to the distribution of primes in
short intervals.  

\begin{conjecture}\label{primeshort}
For any $\eps>0$, if $x$ is large and $y\le x^{1-\epsilon}$, then
$$
\psi(x+y) - \psi(x) = y + o(x^{\frac 12}/\log \log x).
$$
\end{conjecture}

\begin{theorem}\label{theoremshort}
Assume RH.
Conjecture \ref{primeshort} implies Conjecture \ref{conj3}, 
and hence Conjectures \ref{conj1} and \ref{conj2}.  Conversely,
if RH and Conjecture \ref{conj3} holds, then for all fixed $\eps>0$, large $x$
and $y\le x^{1-\eps}$,
$$
\psi(x+y) - \psi(x) = y + o(x^{\frac12} \log x).
$$
\end{theorem}

{\bf Remarks.}  
Whereas the behavior of the left side of \eqref{Landau} is governed by a single
prime when $x$ is small, for larger $x$ the sum is governed by the
primes in an interval.   It has been conjectured (\cite{MS},
Conjecture 2) that for $x^\eps \le h\le x^{1-\eps}$,
$\psi(x+h)-\psi(x)-h$ is normally distributed with mean 0 and variance
$h\log (x/h)$.  Thus, it is reasonable to conjecture that for every $\eps>0$, 
\be\label{psixy}
\psi(x+y)-\psi(x)-y \ll_\eps y^{1/2} x^{\eps} \qquad (1\le y\le x),
\ee
a far stronger
assertion than Conjecture \ref{primeshort}.  It is known that RH implies
$\psi(x)=x+O(x^{1/2}\log^2 x)$ (von Koch, 1900).

A statement similar to the second part of Theorem \ref{theoremshort} has 
been given by Gonek (\cite{Go}, Theorem 4).   Assuming RH, Gonek showed that if
$$
\sum_{0<\g\le T} x^{i\g} \ll_\eps T x^{-1/2+\eps} + T^{1/2} x^{\eps}
$$
holds uniformly for all $x,T\ge 2$ and for each fixed $\eps>0$, then
\eqref{psixy} follows.

\medskip

We also want to describe how to bound the sum $\sum_{0<\g \le T}
x^{i\g}$ in terms of the pair correlation function
\be\label{PCF}
\FF(x,T) = \sum_{0<\g,\g'\le T}
\frac{4 x^{i(\g-\g')}}{4+(\g-\g')^2}.
\ee
Such bounds have been given by Gallagher
and Mueller \cite{GM}, Mueller \cite{Mue},
Heath-Brown \cite{HB}, and Goldston and
Heath-Brown \cite{GH}. 
First we state a strong version of the Pair
Correlation Conjecture for $\zeta(s)$.

\begin{conjecture}\label{conjPC}
Fix a real number $A>1$.  Uniformly for all 
$\frac{T^2}{(\log T)^{6}} \le x\le T^{A}$ we have
$$
\FF(x,T)=N(T)+o\left(\frac{T}{\log T}\right) \qquad (T\to\infty).
$$
\end{conjecture}

\begin{theorem}\label{theoremPC}
Assume RH.  Then Conjecture \ref{conjPC} implies
Conjecture \ref{conj3}, and 
therefore also Conjectures \ref{conj1} and \ref{conj2}. 
\end{theorem}

{\bf Remarks.}  
The original pair correlation conjecture of Montgomery \cite{M1} states that
$$
\FF(x,T) \sim N(T) \qquad (T\to\infty)
$$
uniformly for $T\le x \le T^A$, where $A$ is any fixed real number.
Tsz Ho Chan \cite{Ch} has made an even stronger conjecture than
Conjecture \ref{conjPC}, namely he conjectured that for any
$\epsilon>0$ and any large $A>1$,
$$
\FF(x,T) = N(T) + O\left(T^{1-\epsilon_1}\right)
$$
if $T^{1+\epsilon}\le x\le T^A$, where $\epsilon_1>0$ may depend on $\epsilon$,
and the implicit constant may depend on $\epsilon$ and $A$.  


In the next section, we prove Theorems \ref{theorem1}--\ref{theoremPC}.  In
section \ref{sec:genF} we discuss analogous results for general
$L$-functions.

%
\section{Proof of Theorems \ref{theorem1}--\ref{theoremPC}}
%

\emph{Proof of Theorem \ref{theorem1} (i)}.
Let $\II$ denote an interval of $\T$ for which 
$|\int_\II  g_{\a}(x) dx|$ attains its maximum.  Let $\epsilon$ be a small
positive number, and let
$h_{\epsilon}:\T \to \R$ be a smooth function satisfying 
$h_{\epsilon}(x) \ge 0$ for all 
$x$, $h_{\epsilon}(x)=0$ for $\epsilon <x \le 1$, and 
$\int_\T h_{\epsilon}(x)dx=1$. 
Set $f(x) =\int_{\T} h_{\epsilon}(y)\chi_{\II}(x-y) dy$, where 
${\chi}_\II$ 
denotes the characteristic function of the interval $\II$.  
Then $f$ is smooth, and so 
\eqref{FZ} holds for $f$.  
Therefore 
\be\label{heps}
\int_\T h_{\epsilon}(y) \Big( \sum_{\substack{ 0<\g \le T \\ \{\a \g\} 
\in \II+y}} 1 - N(T) |\II| \Big)\, dy = 
T \int_0^\eps h_{\epsilon}(y) \int_{\II+y} g_{\a}(x) dx\, dy + o(T).
\ee
By \eqref{galpha}, $g_{\a}$ is bounded and it follows that 
$$
\Big| \int_{\II+y}g_{\a}(x) dx -\int_{\II} g_{\a}(x)dx \Big|  
\ll \epsilon
$$
for $0\le y\le \eps$.
Therefore the right side of \eqref{heps} equals 
$$
T \int_\II g_{\a}(x) dx + o(T) +O(\epsilon T).
$$
It follows that for some choice of $y\in (0,\epsilon)$ one must have 
$$
\Big| \sum_{\substack{ 0< \g \le T\\ \{\a \g \} \in \II+y}} 1 -N(T) |\II| 
\Big| \ge T\Big|\int_{\II} g_{\a}(x) dx \Big| + o(T) + O(\epsilon T).
$$
Letting $\epsilon \to 0$, we obtain our lower bound for the discrepancy.
\bigskip

\emph{Proof of Theorem \ref{theorem1} (ii) and Theorem \ref{123}}.
Let 
$$
h(u) = \begin{cases} 1 &  \{ u\} \in \II \\ 0 & \text{else} \end{cases}
$$
and let $J$ be a positive integer.  There are trigonometric
polynomials $h^+$ and $h^-$, depending on $J$ and $\II$, satisfying
\benn
\begin{split}
h^-(u) &\le h(u) \le h^+(u) \qquad (u\in \R), \\
h^{\pm}(u) &= \sum_{|j| \le J} c_j^{\pm} e^{2\pi i j u}, \\
c_0^\pm &= |\II| \pm \frac{1}{J+1}, \qquad |c_j^{\pm}| \le
\frac{1}{|j|} \quad (j\ge 1).
\end{split}
\eenn
For proofs, see Chapter 1 of \cite{M2}, for example.  These
trigonometric polynomials are optimal in the sense that with $J$
fixed, $|c_0^\pm-|\II||$ cannot be made smaller.  We have
$$
\sum_{0<\g\le T} h^-(\a\g) \le 
\sum_{\substack{ 0< \g \le T \\ \{\a \g\} \in \II}} 
1 \le \sum_{0<\g\le T} h^+(\a\g).
$$
For integers $j$, let $x_j=e^{2\pi j \a}$ and for positive $j$ put
$$
V_j = \frac{-\Lambda (n_{x_j})}{2\pi x_j^{1/2}}  
\frac{e^{iT\log(x_j/n_{x_j})}-1}{i\log(x_j/n_{x_j})}.
$$
Also define $V_{-j} = \overline{V_j}$.
By \eqref{Landau}, for nonzero $j$ we have
$$
 \sum_{0<\g\le T} x_j^{i \g} = V_j + O\( x_{|j|}^{1/2} \log^2 (x_{|j|} T) \).
$$
This will be used for 
$$
1 \le |j| \le J_0 := \left\lfloor \frac{2\log T - 5\log\log T}{2\pi \a}
\right\rfloor. 
$$
Suppose that $J\ge J_0$.
We obtain (implied constants depend on $\a$)
\begin{align*}
\sum_{0<\g\le T} & h^{\pm}(\a\g) = c_0^{\pm} N(T) + \sum_{1\le |j| \le
  J} c_j^{\pm} \sum_{0<\g\le T} x_j^{i \g} \\
&=  c_0^{\pm} N(T) + 2 \Re \sum_{1\le j \le J_0}
  c_j^{\pm} \biggl[V_j + O(x_j^{1/2} \log^2 T) \biggr] \\
&\qquad + 
  \sum_{J_0 < |j| \le J} O\pfrac{1}{|j|} \Big| \sum_{0<\g\le T}
  x_j^{i\g} \Big| \\
&\!\!\!\!\!= |\II| N(T)+ \sum_{j\ne 0} c_j^{\pm} V_j \pm \frac{N(T)}{J+1} + o(T) + \!\!
\sum_{J_0<|j| \le J}  O(|j|^{-1}) \Big| \sum_{0<\g\le T}
  x_j^{i\g} \Big|,
\end{align*}
where the term $o(T)$ is uniform in $\II$.
If $\a = \frac{a}{q} \frac{\log p}{2\pi}$ for a prime $p$ and coprime
positive $a,q$, then $x_j = p^{aj/q}$ and consequently 
$$
V_{kq}=-\frac{T\log p}{2\pi p^{ak/2}}
$$
for nonzero integers $k$.
Thus,
$$
\sum_{\substack{j\ne 0 \\ q|j }} c_j^{\pm} V_j = T \int_\T h^{\pm} g_\a.
$$
If $q\nmid j$, then $x_j$ is not an integer.
Hence
$$
\sum_{\substack{j\ne 0 \\ q\nmid j}} c_j^{\pm} V_j \ll T 
\sum_{\substack{1\le |j| \le J \\ q \nmid j}}
\frac{1}{e^{\pi j \a}} \, \left| \frac{e^{iT\log(x_j/n_{x_j})}-1}
{iT\log(x_j/n_{x_j})} \right|.
$$
The sum on the right converges uniformly in $T$, and each summand is
$o(1)$ as $T\to\infty$, hence the left side is $o(T)$.
We conclude
\be\label{sumcjVj}
\sum_{j\ne 0} c_j^{\pm} V_j = T \int_\T h^{\pm} g_\a + o(T).
\ee
When $\a$ is not of the form $\frac{a}{q} \frac{\log p}{2\pi}$, 
$x_j$ is never an integer (for nonzero $j$), and a similar argument 
yields \eqref{sumcjVj}.
Since $h-h^\pm$ has constant sign,
$$
\Big| \int_\T (h-h^{\pm})g_\a \Big| \le \max_{x\in \T} |g_\a(x)| 
\int_\T |h-h^\pm| = \frac{\max_{x\in\T} |g_\a(x)|}{J+1} \ll
\frac{1}{\log T}.
$$
Therefore,
\begin{align*}
\sum_{0<\g\le T} h^{\pm}(\a\g) &= |\II| N(T) + T \int_{\T} h g_\a 
+ o(T) \pm \frac{N(T)}{J+1} \\
&\qquad + \sum_{J_0<|j| \le J}  O\pfrac{1}{|j|} 
\Big| \sum_{0<\g\le T} x_j^{i\g} \Big|.
\end{align*}

For Theorem \ref{theorem1} (ii), we take $J=J_0$.  
For Theorem \ref{123}, take $J=\lfloor \lam \log T
\rfloor$ with $\lam$ fixed, and then let $\lam\to\infty$.

\bigskip

\emph{Proof of Theorem \ref{theoremshort}}.
We first construct a function $F$ which is a good approximation of the
characteristic function of the interval $[0,1]$ and whose Fourier
transform is supported on $[-K,K]$, where $K$ is a parameter to be
specified later.  Consider the entire function
$$
H(z)=\pfrac{\sin \pi z}{\pi}^2 \biggl( \sum_{n=1}^\infty \frac{1}{(z-n)^2} - 
\sum_{n=1}^\infty \frac{1}{(z+n)^2} + \frac{2}{z} \biggr)
$$
for complex $z$, and set
$$
F(z) = \frac{H(Kz) + H(K-Kz)}{2}.
$$
The function $H(z)$ is related to the so-called Beurling-Selberg
functions, and basic facts about $H$ can be found in \cite{V}.
In particular, for real $x$, (i) $H(x)$ is an odd function; (ii) the
Fourier transform $\hatH$ is supported on $[-1,1]$; (iii) $H(x) =
\sgn(x) + O(\frac{1}{1+|x|^3})$, where $\sgn(x)=1$ if $x>0$,
$\sgn(x)=-1$ if $x<0$ and $\sgn(0)=0$; (iv) $H'(x)
=O(\frac{1}{1+|x|^3})$.  Item (iii) follows from (2.26) of \cite{V}
and the Euler-Maclaurin summation formula, and (iv) follows from
Theorem 6 of \cite{V}.  Let $I$ be the indicator function 
of the interval $[0,1]$.  It follows that the Fourier transform
$\hatF$ of $F$ is supported on $[-K,K]$ and
\be\label{Fest}
| F(x) - I(x) | \ll \frac{1}{1+K^3 |x|^3} + \frac{1}{1+K^3|1-x|^3}.
\ee
Since 
$$
\hatI(t) = \frac{1-e^{-2\pi i t}}{2\pi i t},
$$
it follows readily that $\hatF(t) \ll 1$, uniformly in $K$, and
\begin{align*}\label{Fhatprime}
\hatF'(t) &= \frac{1-(1+2\pi i t) e^{-2\pi i t}}{-2\pi i t^2} + 
O\biggl( \int_{-\infty}^{\infty} \frac{|x|}{1+K^3 |x|^3} +
\frac{|x|}{1+K^3 |1-x|^3}\, dx \!\biggr) \\
&= O\( \frac{1}{1+|t|} + \frac{1}{K} \).
\end{align*}

Next, let $T\ge 2$ and $T \le x \le T^A$.   Write
$$
\sum_{0<\g\le T} x^{i\g} = \sum_{|\g| \le x} x^{i\g} F(\g/T) +
\sum_{|\g| \le x} x^{i\g} \bigl[ I(\g/T)-F(\g/T) \bigr].
$$
By \eqref{NT} and \eqref{Fest}, the second sum on the right is
\begin{align*}
&\ll N\pfrac{T}{K} + \( N\(T + \frac{T}{K}\) -
N\(T - \frac{T}{K}\)\) \\
&\qquad + \frac{T^3}{K^3} \biggl( \; \sum_{|\g| > T/K}
\frac{1}{|\g|^3} + \sum_{|\g-T| \ge T/K}
\frac{1}{|\g-T|^3} \; \biggr) \\
&\ll \frac{T\log T}{K}.
\end{align*}
Also,
\begin{align*}
\sum_{|\g| \le x} x^{i\g} &F(\g/T) \sum_{|\g| \le x} x^{i\g}
  \int_{-K}^K e^{2\pi i v \g/T} \hatF(v) \, dv \\
&= x^{-1/2} \int_{-K}^K e^{-\pi v/T} \hatF(v) \sum_{|\g|\le x} \( x
  e^{2\pi v/T} \)^\rho\, dv \\
&= -\frac{T}{2\pi x^{1/2}} \int_{-K}^K e^{-\pi v/T} \( \hatF'(v) - \frac{\pi}{T}
  \hatF(v) \) \sum_{|\g| \le x} \frac{\( x e^{2\pi v/T}
  \)^\rho}{\rho}\, dv,
\end{align*}
where the last line follows from the previous line using integration
by parts.  The final sum on $\g$ is evaluated using the explicit
formula (see e.g. \cite{Da}, \S 17)
\be\label{explicit}
G(x) := \psi(x) - x = -\sum_{|\g| \le M} \frac{x^\rho}{\rho}
+ O\( \log x + \frac{x\log^2 (Mx)}{M} \),
\ee
valid for $x\ge 2$, $M\ge 2$.  
Since 
$$
\int_{-K}^K e^{-\pi v/T} \(\hatF'(v)-\frac{\pi}{T} \hatF(v) \)\, dv =0,
$$
we obtain
\begin{align*}
\sum_{|\g| \le x} x^{i\g} F(\g/T) &= \frac{-T}{2\pi \sqrt{x}}
\int_{-K}^K \hatF'(v) \( G(xe^{2\pi v/T})-G(x) \)\, dv \\
&\qquad + O\( K \( 1 + T x^{-1/2} \) \log^2 x\).
\end{align*}
Altogether, this gives
\begin{align*}
\sum_{|\g| \le T} x^{i\g} &\ll \frac{T\log K}{\sqrt{x}} \max_{xe^{-2\pi
    K/T} \le y \le x e^{2\pi K/T}} | G(y)-G(x) | \\
&\qquad + \frac{T\log T}{K} +
K \( 1 + T x^{-1/2} \) \log^2 x.
\end{align*}
Take $K=\log^2 T$ and assume Conjecture \ref{primeshort}.
The first part of Theorem \ref{theoremshort} follows.

The second part is straightforward, starting with the explicit formula
\eqref{explicit} in the form
$$
\psi(x+y)-\psi(x)-y = - \sum_{|\g| \le x} \frac{(x+y)^\rho -
  x^\rho}{\rho} +O(\log^2 x).
$$
Fix $\eps>0$ and apply Conjecture \ref{conj3} with $A=2/\eps$.  By
partial summation,
\begin{align*}
\Big|\sum_{x^{\eps/2} < |\g| \le x} \frac{x^{\rho}}{\rho} \Big| 
&=2 \Big| \Re \sum_{x^{\eps/2} < \g \le x} \frac{x^{\rho}}{\rho} \Big| \\
&\le 2 x^{1/2}
  \biggl| \frac{1}{\frac12 + ix} \sum_{0<\g\le x} x^{i\g} +
   i \int_{x^{\eps/2}}^x
  \frac{1}{(\frac12+it)^2} \sum_{0<\g\le t} x^{i\g} \, dt \biggr|\\
&= o(x^{1/2} \log x).
\end{align*}
The smaller zeros are handled in a trivial way.  We have, for $y\le
x$, 
$$
(x+y)^\rho - x^\rho = x^\rho \( \rho \frac{y}{x} + O\pfrac{|\rho|^2
  y^2}{x^2} \),
$$
whence
$$
\sum_{|\g| \le x^{\eps/2}} \frac{(x+y)^\rho - x^\rho}{\rho} \ll
N(x^{\eps/2}) x^{1/2} \( \frac{y}{x} + x^{\eps/2} \frac{y^2}{x^2} \)
\ll x^{\frac{1}{2} - \frac{\eps}{2}} \log x.
$$
Therefore, $\psi(x+y)-\psi(x) - y = o(x^{1/2}\log x)$, as claimed.

\bigskip

\emph{Proof of Theorem \ref{theoremPC}}.
It will be convenient to work with the normalized sum
$$
\DD(x,T) = \frac{\FF(x,T)}{N(T)}.
$$

\begin{lemma}\label{PCY}
Suppose $T\ge 10$ and $1\le\b\le \frac{T}{2\log T}$.  Then
\begin{align*}
\sum_{0<\g \le T} &x^{i\g} \ll T \pfrac{\log T}{\b}^{\frac12} \! \biggl( 1 + 
\max_{\frac{T}{\b\log T}\le t\le T}
|\DD(x,t)| \\
&\qquad + \b^3 \biggl|\int_{-\infty}^\infty (\DD(xe^{u},t)-\DD(x,t))
e^{-2\b|u|} \, du \biggr| \biggr)^{\frac12} \\
&\ll \frac{T(\log T)^{\frac12}}{\b^{1/2}} \( 1 +
\max_{\frac{T}{\b\log T}\le t\le T}
|\DD(x,t)| \)^{1/2} 
 +T (\b \log T)^{1/2} \\
&\quad \times \biggl( 
\max_{\frac{T}{\b\log T}\le t\le T} \; \max_{0\le u \le \frac{1}{\b}\log
  (\b\log T)} |\DD(xe^u,t)+\DD(xe^{-u},t)-2\DD(x,t)| \biggr)^{\frac12}.
\end{align*}
\end{lemma}

\begin{proof}  
We follow \cite{GH} by estimating $\sum_{0<\g \le T} x^{i\g}$ in terms
of
$$
G_\beta(x,T) =  \sum_{0<\g,\g'\le T}
\frac{4\beta^2 x^{i(\g-\g')}}{4\beta^2+(\g-\g')^2}.
$$
In particular, $G_1(x,T) = \FF(x,T)$, and
by \eqref{NT}, we have $G_\b(x,T) \ll (1+\b) T \log^2 T$.
By Lemma 1 of \cite{GH}, uniformly for $1\le \b \le T$ and $1\le V\le T$, 
we have
\be\label{YG}
\begin{split}
\sum_{0<\g \le T} x^{i\g} &\ll 
  \Big(T\b^{-1} \max_{t\le T}G_\b(x,t)\Big)^{1/2} \\ &\ll
  \frac{T\log T}{V^{1/2}} + \Big(T\b^{-1} \max_{T/V \le t
  \le T} G_\b(x,t) \Big)^{1/2}.
\end{split}
\ee
 Using Lemma 2 of \cite{GH}, we have
\begin{align*}
G_\b(x,t) &= \b^2 \FF(x,t) +  \b(1-\b^2) \int_{-\infty}^\infty \FF(xe^{u},t)
e^{-2\b|u|}\, du \\
&= \FF(x,t) +  \b(1-\b^2) \int_{-\infty}^\infty (\FF(xe^u,t)-\FF(x,t))
e^{-2\b|u|}\, du,
\end{align*}
from which the first inequality in the lemma follows upon taking 
$V=\b \log T$.  For the second
inequality, combine the terms in the integral with $u=v$ and $u=-v$ for
$0\le v \le \frac{\log (\b \log T)}{\b}$,
and use the trivial bound $\DD(z,t) \ll \log t$ when $|u| \ge
\frac{\log (\beta \log T)}{\beta}$ ($z=x$ and $z=xe^u$).
\end{proof}

In order to finish the proof of Theorem \ref{theoremPC}, 
suppose that $\log T \le \b \le \log^2 T$.  From
Conjecture \ref{conjPC} it follows that the terms
$\DD(xe^u,t)$, $\DD(xe^{-u},t)$, and $\DD(x,t)$, in the
ranges from the statement of the above lemma, are all
of the form $1+o\left( (\log T)^{-2}\right)$. Therefore,
$$
\sum_{0<\g \le T} x^{i\g}=
O\left(T\frac{(\log T)^{1/2}}{\beta^{1/2}}\right)
+o\left(T\frac{\beta^{1/2}}{(\log T)^{1/2}}\right).
$$
Thus, taking $\beta$ slightly larger than $\log T$
produces the desired result.

%
%
\section{General $L$-functions}\label{sec:genF}
%

Consider a Dirichlet series $F(s)=
\sum_{n=1}^\infty a_F(n) n^{-s} $ satisfying the following axioms:
\par
\noindent (i) there exists an integer $m \geq 0$ such that $(s-1)^mF(s)$
is an entire function of finite order; 
\par \noindent (ii)  $F$ satisfies
a functional equation of the type: 
$$
\Phi(s) = w\overline{\Phi}(1-s), 
$$
where 
$$
\Phi(s) = Q^s \prod_{j=1}^r \Gamma(\lambda_j s + \mu_j) F(s) 
$$
with $Q>0$, $\lambda_j > 0$, $\Re(\mu_j) \geq 0$ and $|w|=1$.  (Here,
$\overline{f}(s) = \overline{f(\overline{s})}$); 
\par \noindent (iii) $F(s)$ has an Euler product, which we write as
$$
-\frac{F'}{F}(s)
=\sum_{n=1}^\infty \Lambda_F(n) n^{-s},
$$
where $\Lambda_F(n)$ is supported on powers of primes.

We also need some growth conditions on the coefficients $a_F(n)$ and
$\Lambda_F(n)$.  Although stronger than we require, for convenience 
we impose the
conditions (iv) $\Lambda_F(n) \ll n^{\theta_F}$ for some $\theta_F
< \frac12$ and (v) for every $\eps>0$, $a_F(n)\ll_\eps n^{\eps}$.
Together, conditions (i)--(v) define the \emph{Selberg class} $\SS$ of
Dirichlet series.
For a survey of results and conjectures concerning the
Selberg class, the reader may consult Kaczorowski and Perelli's paper
\cite{KP}.  In particular, $\SS$ includes the Riemann zeta
function, Dirichlet $L$-functions, and $L$-functions attached to
number fields and elliptic curves.  The Selberg class
is conjectured to equal the class of all automorphic $L$-functions,
suitably normalized so that their nontrivial zeros have real parts
between 0 and 1.

The functional equation is not uniquely determined in light of the
duplication formula for $\Gamma$-function, however the real sum
$$
d_F = 2 \sum_{j=1}^r \lam_j
$$
is well-defined and is known as the degree of $F$.  Analogous to
\eqref{NT}, we have (cf. \cite{S3}, (1.6))
\be\label{NFT}
\begin{split}
N_F(T) &= \left| \{ \rho=\beta+i\gamma : F(\rho)=0, 0<\beta<1,
0<\gamma\le T\} \right| \\
&= \frac{d_F}{2\pi} T \log T + c_1 T + O(\log T)
\end{split}
\ee
for some constant $c_1=c_1(F)$.
A function $F\in \SS$ is said to be \emph{primitive} if it cannot be
written as a product of two or more elements of $\SS$. 
We henceforth assume that $F$ is primitive.  The extension of our
results to non-primitive $F$ is straightforward.
It is expected that all zeros of $F$ with
real part between 0 and 1 have real part $\frac12$, a hypothesis we
abbreviate as \RHF.  Although we shall assume \RHF\,  for many of the
results in this section, sometimes a weaker hypothesis suffices, that
most zeros of $F$ are close to the critical line.

\ms
\noindent {\bf Hypothesis} $Z_F$. There exist constants $A>0, B>0$
 (depending on $F$) such that 
\begin{align*}
N_F(\sigma,T) &= \left| \left\{ \beta+i\gamma: \frac12 \le \beta \le \sigma,
 0<\gamma \le T \right\} \right| \\ &\ll T^{1-A(\sigma-1/2)}\log^B T,
\end{align*}
uniformly for $\sigma\ge 1/2$ and $T\ge 2$.
\ms

Hypothesis $Z_F$ is known, with $B=1$, for the Riemann zeta
function and Dirichlet $L$-functions (Selberg \cite{S1}, \cite{S2}),
and certain degree 2 $L$-functions attached to cusp forms (Luo \cite{Luo}).

The next tool we require is an analog of \eqref{Landau}.
It is very similar to Proposition 1 of \cite{MZ}, and with small
modifications to that proof we obtain the following result,
which is nontrivial provided $x^{1/2+\th_F}+x^{1/2+\eps} \ll T$.

\begin{lemma}\label{lem1} Let $F\in\mathcal S$, $x>1$, $T\geq2,$ and 
let $n_x$ be a nearest integer to $x$.
Then, for any $\e > 0$,
\begin{align*}
\sum_{0<\gamma\leq T}x^{\rho} &=
-\frac{\Lambda_F(n_x)}{2\pi} \frac{e^{iT\log(x/n_x)}-1}
{i\log(x/n_x)} \\
&\qquad +O_{\e}\(x^{1+\theta_F}\log (2x) + x^{1+\e} \log T + \frac{\log T}{\log
  x} \).
\end{align*}
\end{lemma}

Using Lemma \ref{lem1} in place of Lemma 1 of \cite{FZ}, Hypothesis
$Z_F$ in place of Lemma 2 of \cite{FZ}, and following the proof of
Theorem 1 of \cite{FZ}, we obtain a generalization of \eqref{FZ}.

\begin{theorem}\label{theorem0F}
Let $F\in \SS$.  
If $\alpha= \frac{a \log p}{2\pi q}$ for some prime number $p$
and positive integers $a,q$ with $(a,q)=1$, define
$$
g_{F,\alpha}(t)= -\frac1{\pi} \Re \sum_{k=1}^\infty
\frac{\Lambda_F(p^{ak})}{p^{ak/2}}e^{-2\pi i qkt}.
$$
For other $\a$, define $g_{F,\a}(t)=0$ for all $t$.
If Hypothesis $Z_F$ holds, then 
\be\label{uniformF}
\sum_{0< \g \le T} f(\a\g)  = N_F(T) \int_\T f(x) \, dx
+ T \int_\T f(x) g_{F,\a}(x)\, dx +  o(T)
\ee
for all $f \in C^2(\T)$.  Assuming \RHF, \eqref{uniformF} holds for all
absolutely continuous $f$.
\end{theorem}

Since Hypothesis $Z_F$ holds for Dirichlet $L$-functions $L(s,\chi)$,
we obtain the following.

\begin{corollary}\label{Dirichlet}
Unconditionally, for Dirichlet $L$-functions $F$, \eqref{uniformF} holds
for all $f \in C^2(\T)$. 
\end{corollary}
  
When $F(s)=L(s,\chi)$ and $\a = \frac{a\log p}{2\pi q}$ with $p$ prime,
$(a,q)=1$, we have
$$
g_{F,\alpha}(t)= -\frac{\log p}{\pi} \Re \left(
\frac{e^{2\pi i(qt+a\xi)}}{p^{a/2}-e^{2\pi i(qt+a\xi)}}\right),
$$
where $\chi(p)=e^{2\pi i \xi}$.
It follows that there is a shortage of zeros of $L(s,\chi)$
with $\{\alpha\g\}$ near $\frac{k-a\xi}{q}$, $k=0,\cdots,q-1$.
We illustrate this phenomenon
with three histograms of $M_F(y;T)$, where
$$
M_F(y) = \frac{T}{N_F(T)}
\Bigg| \sum_{\substack{ 0< \g \le T \\ \{\a \g\} < y}} 
1 - y  N_F(T) \Bigg|,
$$
$F$ a Dirichlet $L$-function
associated with a character of conductor 5 and
$T=500,000$.  For both characters, $N_F(T)=946488$.  The list of zeros
was taken from Michael Rubinstein's data files on his Web page.
In Figure 1 we plot for each subinterval $I=[y,y+\frac{1}{500})$ the value
of $500 (M_F(y+\frac{1}{500})-M_F(y))$ and also the graph of
$g_{F,\a}(y)$.
The characters are identified by their value at 2.

%
%

\afterpage{\clearpage}
\begin{figure}[t]
\parbox{.8\linewidth}{\subfigure{
\epsfig{file=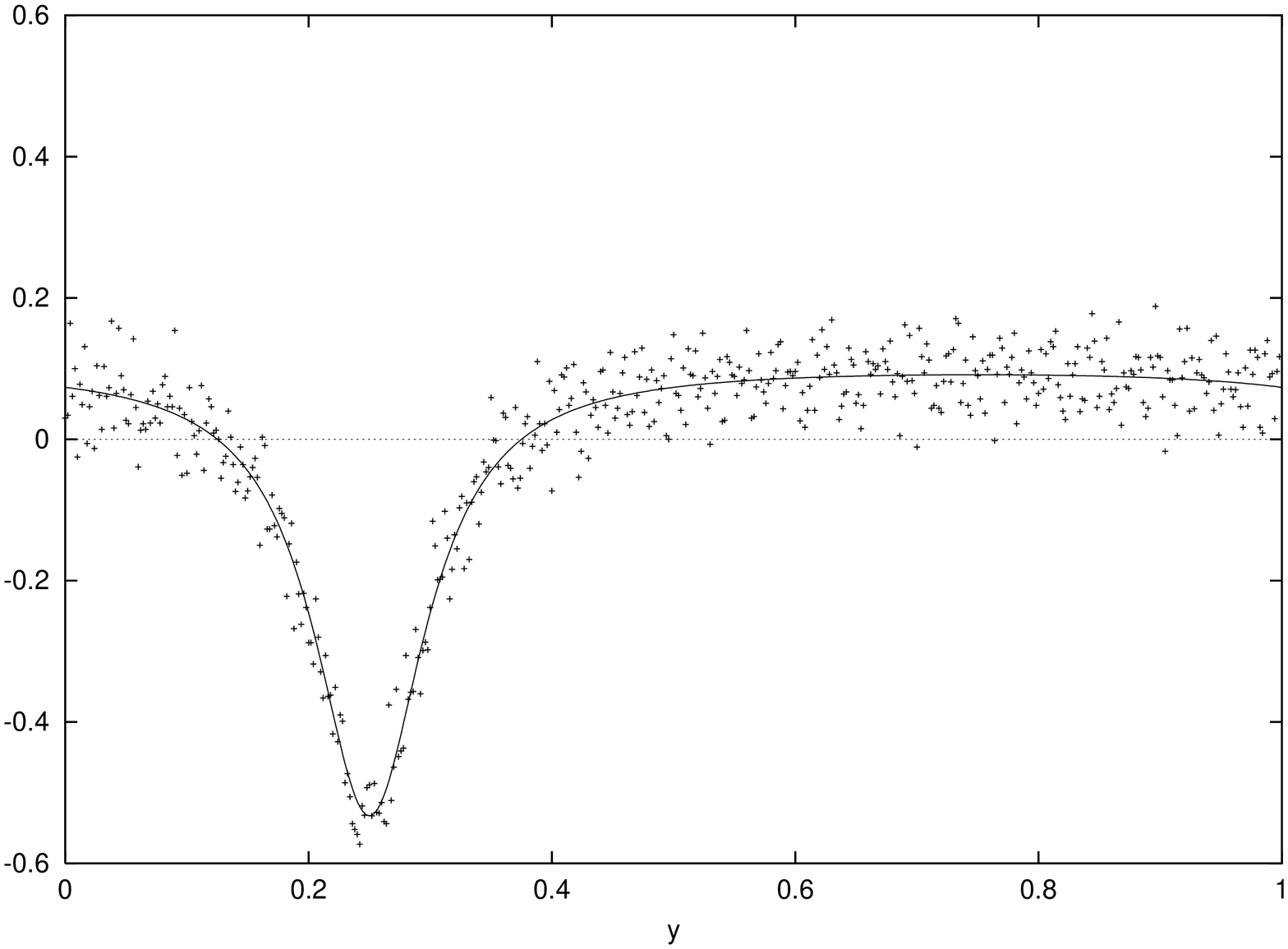, height=3.0in, width=1.8in, angle=270}}}
\nts\nts\nts\nts\nts\nts\nts\nts$\ds \a=\tfrac{\log 2}{2\pi}, \chi(2)=-i$ 
\parbox{.8\linewidth}{\subfigure{
\epsfig{file=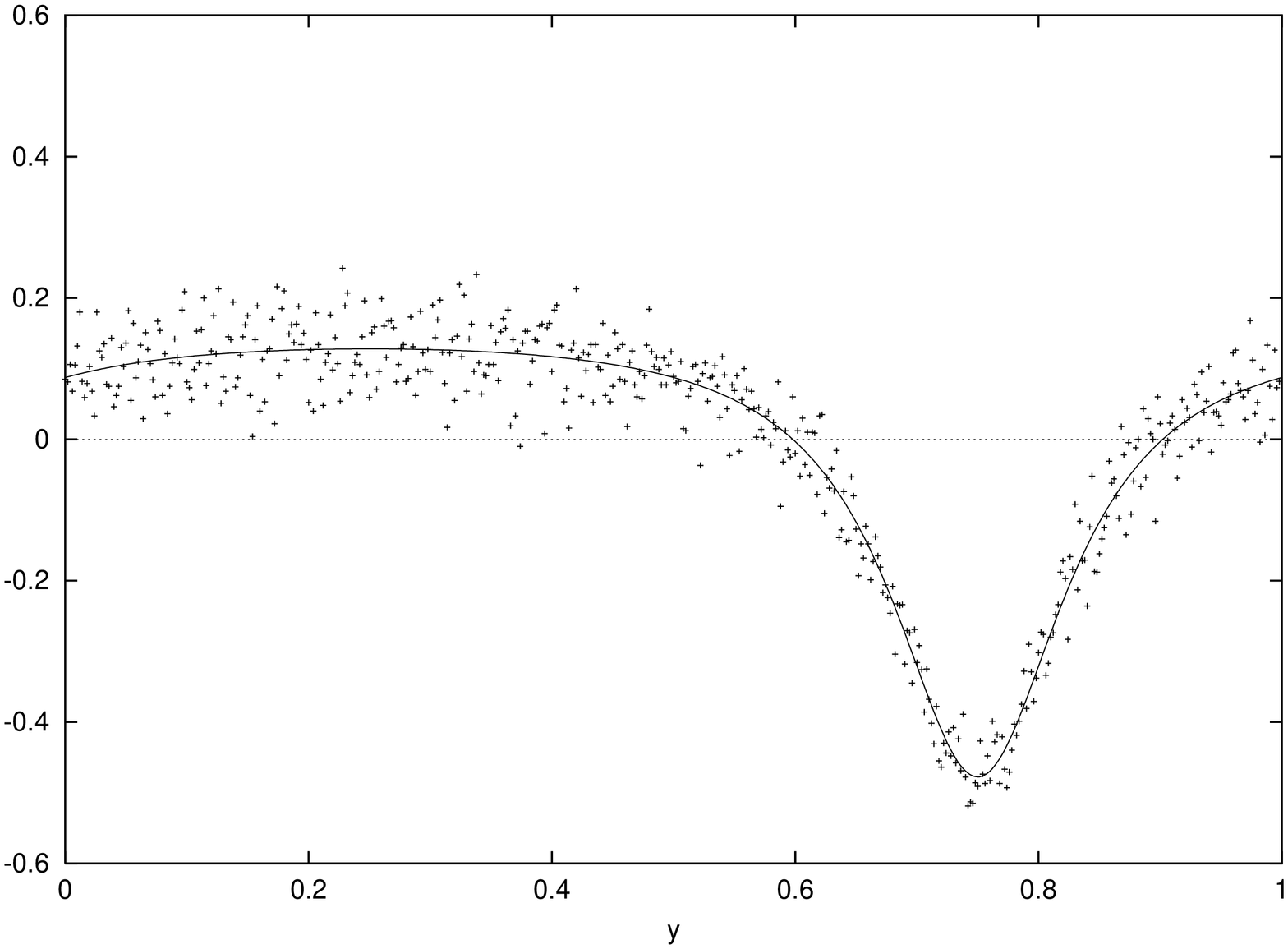, height=3.0in, width=1.8in, angle=270}}}
\nts\nts\nts\nts\nts\nts\nts\nts$\ds \a=\tfrac{\log 3}{2\pi}, \, \chi(2)=-i$
\parbox{.8\linewidth}{\subfigure{
\epsfig{file=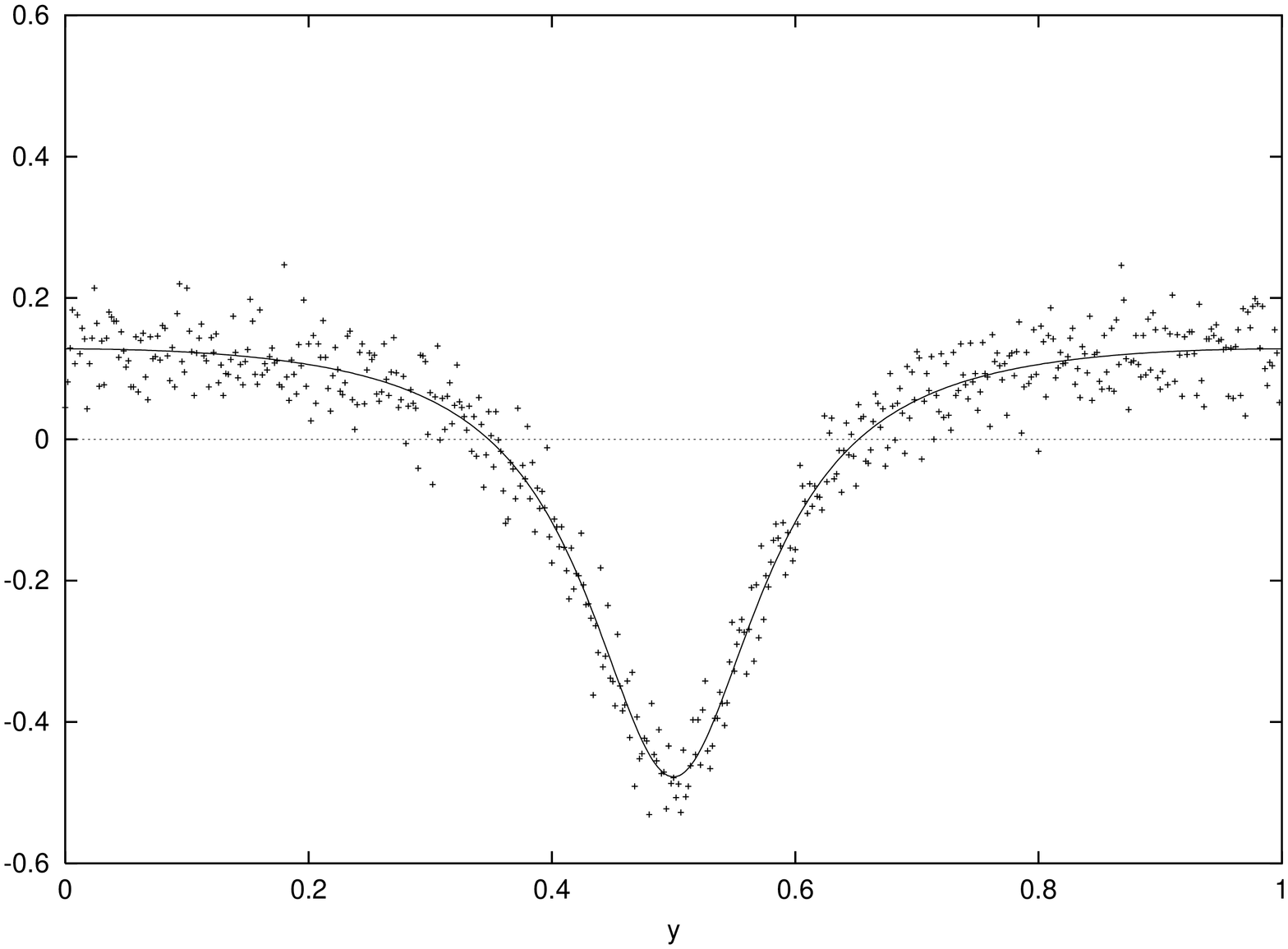, height=3.0in, width=1.8in, angle=270}}}
\nts\nts\nts\nts\nts\nts\nts\nts$\ds \a=\tfrac{\log 3}{2\pi}, \, \chi(2)=-1$
\caption{$500(M_F(y+\frac{1}{500})-M_F(y))$ vs. $g_{F,\a}(y)$ for
$T=500000$.}
\end{figure}

We conjecture that \eqref{uniformF} holds when $f$ is the indicator
function of an interval, and are thus led to the following
generalizations of Conjectures \ref{conj1} and \ref{conj2}.
Here $D_{F,\a}$ is the natural generalization of the discrepancy
function $D_\a$.

\begin{conjecture}\label{conj1F}
Let $\II$ be an interval of $\T$.  Then 
$$
\sum_{\substack{ 0< \g \le T \\ \{\a \g\} \in \II}} 
1 = |\II| N_F(T) + T \int_{\II} g_{F,\a}(x) dx + o(T).
$$
\end{conjecture}

\begin{conjecture}\label{conj2F} 
We have 
$$
D_{F,\a}(T) = \frac{T}{N_F(T)} \sup_{\II} \Big| \int_{\II} g_{F,\a}(x)
\, dx \Big| + o\Big(\frac{1}{\log T} \Big).
$$
\end{conjecture}

Combining Theorem \ref{theorem0F} and the proof of Theorem \ref{theorem1},
we obtain the following.  The only difference in the proof is that
here we take
$$
J_0 = \left\lfloor \frac{\tfrac{\log T}{1/2+\theta_F}-5\log\log T}{2\pi \a} 
\right\rfloor.
$$

\begin{theorem}\label{theorem1F}
(i) Assuming Hypothesis $Z_F$, we have 
$$
D_{F,\a}(T) \ge \frac{T}{N_F(T)}  \sup_{\II} \Big| \int_{\II}
g_{F,\a}(x) \, dx \Big|  + o\Big(\frac{1}{\log T} \Big).
$$

(ii) Assuming \RHF, for any interval $\II$ of $\T$ we have 
$$
\Big| \sum_{\substack{ 0<\g \le T \\ \{\a \g\} \in \II}} 1 
- |\II| N_F(T) - T\int_{\II} g_{F,\a}(x) dx \Big| \le  
\a(1/2+\theta_F) T + o(T).
$$
\end{theorem}

We can prove a direct analog of Theorem \ref{123}, by requiring a
slightly larger range of $T$ in the analog of Conjecture \ref{conj3},
since $\theta_F$ may be large.

\begin{conjecture}\label{conj3F}
Let $A >1$ be a fixed real number.  Uniformly for 
$$
\frac{T^{1/(1/2+\theta_F)}}{\log^{5} T} \le x\le T^{A},
$$ 
we conjecture that
\be\label{sumxgF}
\sum_{0<\g \le T} x^{i\g} = o(T). 
\ee
\end{conjecture}

\begin{theorem}\label{123F}
Assume \RHF.  Then Conjecture \ref{conj3F} implies Conjectures
\ref{conj1F} and \ref{conj2F}.
\end{theorem}

The analog of Theorem \ref{theoremshort} holds for $F\in \SS$, by
following the proof given in the preceding section.  Here
we need an explicit formula similar to \eqref{explicit}.  By standard
contour integration methods, one obtains
$$
G_F(x) := \sum_{n\le x} \Lambda_F(n) - d_F x = - \sum_{|\rho| \le Q}
\frac{x^\rho}{\rho} + O(x^{\theta_F}\log x) 
$$
provided $Q \ge x\log x$.  Since $\theta_F<\frac12$, the error term is
acceptable. 

\begin{conjecture}\label{primeshortF}
For every $\eps>0$, if $x$ is large and $y\le x^{1-\epsilon}$, then
$$
G_F(x+y) - G_F(x) = o(x^{\frac 12}/\log \log x).
$$
\end{conjecture}

\begin{theorem}\label{theoremshortF}
Assume \RHF.
Conjecture \ref{primeshortF} implies Conjecture \ref{conj3F}, 
and hence Conjectures \ref{conj1F} and \ref{conj2F}.  Conversely,
if \RHF\, and Conjecture \ref{conj3F} holds, 
then for all fixed $\eps>0$, large $x$
and $y\le x^{1-\eps}$, 
$$
G_F(x+y) - G_F(x) = o(x^{\frac12} \log x).
$$
\end{theorem}

In order to address an analog of Theorem \ref{theoremPC}, we first
quote a Pair Correlation Conjecture for $F$,
due to Murty and Perelli \cite{MP}.

\begin{conjecture}\label{conjPCF}
Define
$$
\FF_F(x,T) = \sum_{0<\g,\g'\le T}
\frac{4 x^{i(\g-\g')}}{4+(\g-\g')^2}
$$
and $\DD_F(x,T)=\FF_F(x,T)/N_F(T)$.  We have 
$\DD_F(T^{\th d_F},T) \sim \th$ for $0<\th\le 1$ and
$\DD(T^{\th d_F},T) \sim 1$ for $\th \ge 1$.
\end{conjecture}

Notice that, as a function of $x$, $\FF_F(x,T)$ is conjectured to
undergo a change of behavior in
the vicinity of $x=T^{d_F}$.  In order to deduce Conjecture
\ref{conj3F}, we can postulate a stronger version of Conjecture \ref{conjPCF},
with error terms of relative order $o(1/\log^2 T)$.  
We succeed, as in the proof of Theorem \ref{theoremPC}, when $d_F=1$.
When $d_F \ge 2$, however, this transition zone lies outside
the range in which Lemma \ref{lem1} is useful (Kaczorowski and Perelli
recently proved that $1<d_F<2$ is impossible \cite{KP2}; it is conjectured that
$d_F$ is always an integer).  We can use
an analog of Lemma \ref{lem1}, which follows by the same method (replace
$\DD(x,T)$ with $\DD_F(x,T)$).  However, in order to prove the right side is
small, we require that $\DD_F(x,T)$ has small \emph{variation}, even
through the transition zone $x\approx T^{d_F}$.
Tsz Ho Chan \cite{Ch} studied the
behavior of $\DD(x,T)$ (for $\zeta(s)$) in the vicinity of $x=T$
assuming RH plus a quantitative version of the twin prime conjecture
with strong error term.  His analysis leads to a pair correlation
conjecture with $\DD(x,T)$ smoothly varying through the transition zone.
We conjecture that the same holds for other $F\in \SS$.

\begin{conjecture}\label{conjPCF2} 
For $F\in \SS$, $\DD_F(x,T)\ll 1$ uniformly in $x$ and $T$, 
and for any $A>0$ there is a $c>0$ so that
$$
| \DD_F(x+\delta x,T) + \DD_F(x-\delta x,T) - 2 \DD_F(x,T)| = o (T/\log T)
$$
uniformly for $T \le x\le T^A$ and $0 \le \delta \le (\log T)^{c-1}$.
\end{conjecture}

Following the proof of Theorem \ref{theoremPC} (take $\beta=\log T
\log\log T$, for example), we arrive at the following.

\begin{theorem}\label{theoremPCF}
Assume \RHF.  Then Conjecture \ref{conjPCF2} implies 
Conjecture \ref{conj3F}, and 
therefore also Conjectures \ref{conj1F} and \ref{conj2F}. 
\end{theorem}

\bigskip
{\bf Acknowledgement.}  The authors thank the referee for carefully
 reading the paper and for pointing out several misprints and minor 
errors.


\end{document}